\documentclass[10pt,a4paper]{amsart}
\usepackage{amssymb,amsmath,amsfonts}

\headsep=1.2500cm \numberwithin{equation}{section}
\newcommand{\set}[1]{\left\{#1\right\}}

\newtheorem{Theorem}{Theorem}[section]
\newtheorem{Proposition}[Theorem]{Proposition}
\newtheorem{cor}[Theorem]{Corollary}

\theoremstyle{remark}

\begin{document}

\title{Approximate biprojectivity and $\phi$-biflatness of certain Banach algebras}

\author[A. Sahami]{A. Sahami}

\address{Faculty of Mathematics and Computer Science,
Amirkabir University of Technology, 424 Hafez Avenue, 15914
Tehran, Iran.}

\email{amir.sahami@aut.ac.ir}

\author[A. Pourabbas]{A. Pourabbas}
\email{arpabbas@aut.ac.ir}




\begin{abstract}
In the first part of the  paper, we investigate the approximate biprojectivity  of some Banach algebras related to the locally compact groups. We show that a Segal algebra $S(G)$ is approximately biprojective if and only if $G$ is compact. Also for every continuous weight $w$, we show that $L^{1}(G,w)$ is  approximately biprojective if and only if $G$ is compact.

In the second part, we study $\phi$-biflatness of some Banach algebras, where $\phi$ is a  Banach algebra character. We show that if $S(G)$ is $\phi$-biflat, then $G$ is an amenable group for every character $\phi$. Finally we show that $\phi$-biflatness of $L^{1}(G)^{**}$ implies the amenability of $G$.
\end{abstract}
\maketitle

\section{Introduction and Preliminaries}
The concepts of $\phi$-biflatness, $\phi$-biprojectivity,
$\phi$-Johnson amenability and other related concepts were
introduced and studied in \cite{sah1}. The studies include
determining when the various classes of Banach algebras are, or are
not $\phi$-biflat or $\phi$-biprojective. It was shown in
\cite{sah1} that $L^{1}(G)$ is $\phi$-biflat if and only if $G$ is
an amenable group and the Fourier algebra $\mathcal{A}(G)$ is
$\phi$-biprojective if and only if $G$ is a discrete group.

Recently the concepts of approximate biprojectivity and approximate biflatness have been studied
by  Zhang \cite{zhang} and   Samei {\it et al.} \cite{sam}, respectively.   Samei {\it et al.} in \cite{sam} studied
 approximate biflatness  of Segal algebras and Fourier
algebras and  they showed that the Segal algebra $S(G)$ is pseudo-contractible if and only if $G$ is compact. Note that the pseudo-contractility of Banach algebras implies the approximate
biprojectivity \cite[Proposition 3.8]{ghah pse}, that is, the
approximate biprojectivity is  a weaker notion than the pseudo-contractility, for more details see \cite{ghah pse}.

Motivated by these results, in this paper we  extend \cite[Theorem
3.5]{sam} or \cite[Theorem 5.3]{Choi} and we  show that  Segal
algebra $S(G)$ is approximately biprojective if and only if $G$ is
compact.  The group algebra $L^{1}(G)$ is biprojective if and only
if $G$ is compact, see \cite[Theorem 5.13]{hel}. Here we extend this
result, we show that the weighted group algebra $L^{1}(G,w)$ is
approximately biprojective if and only if $G$ is compact for every
continuous weight $w$ on $G$. We show that if Segal algebra $S(G)$ is
$\phi$-biflat, then $G$ is amenable, where $\phi$ is  any
  character on $S(G)$ and if $L^{1}(G)^{**}$ is
$\tilde{\phi}$-biflat, then $G$ is amenable, where $\tilde{\phi}$ is
an extension of character $\phi$ on $L^{1}(G)$.

We remark some standard notations and definitions that we shall need
in this paper. Let $A$ be a Banach algebra. If $X$ is a Banach
$A$-bimodule, then  $X^{*}$ is also a Banach $A$-bimodule via the
following actions
$$(a\cdot f)(x)=f(x\cdot a) ,\hspace{.25cm}(f\cdot a)(x)=f(a\cdot x ) \hspace{.5cm}(a\in A,x\in X,f\in X^{*}). $$

Throughout, the
character space of $A$ is denoted by $\Delta(A)$,  that is, all
non-zero multiplicative linear functionals on $A$. Let $\phi\in
\Delta(A)$. Then $\phi$ has a unique extension   $\tilde{\phi}\in\Delta(A^{**})$
 which is defined by $\tilde{\phi}(F)=F(\phi)$ for every
$F\in A^{**}$.

Let $A$ be a  Banach algebra. The projective tensor product
 $A\otimes_{p}A$ is a Banach $A$-bimodule via the following actions
$$a\cdot(b\otimes c)=ab\otimes c,~~~(b\otimes c)\cdot a=b\otimes
ca\hspace{.5cm}(a, b, c\in A).$$

 The product morphism
$\pi_{A}:A\otimes_{p}A\rightarrow A$ is  specified by $\pi_{A}(a\otimes
b)=ab$ for every $a,b\in A$.

Let $G$ be a locally compact group. The Fourier algebra on  $G$ is
denoted by $\mathcal{A}(G)$.  It is well-known that the character
space $\Delta(\mathcal{A}(G))$ consists of all point evaluation maps
$\phi_{t}:\mathcal{A}(G)\rightarrow \mathbb{C}$ such that
$\phi_{t}(f)=f(t)$ for each $f\in \mathcal{A}(G)$, see \cite{Eym}.

We also remind some concepts of Banach homology  which we shall need
in this paper. A Banach algebra $A$ is called biprojective, if there
exists a bounded $A$-bimodule morphism $\rho:A\rightarrow
A\otimes_{p}A$ such that $\rho$ is a right inverse for $\pi_{A}$
\cite{hel}. We recall that $A$ is an approximately biprojective
Banach algebra if there exists a net of bounded   $A$-bimodule
morphism $(\rho_{\alpha}):A\rightarrow A\otimes_{p}A$ such that
$\pi_{A}\circ\rho_{\alpha}(a)\rightarrow a$ for each $a\in A,$ see
\cite{zhang}. A Banach algebra  $A$ is called $\phi$-biflat for every $\phi\in\Delta(A)$,
if there exists a bounded $A$-bimodule morphism $\rho:A\rightarrow
(A\otimes_{p}A)^{**}$ such that
$\tilde{\phi}\circ\pi^{**}_{A}\circ\rho(a)=\phi(a)$
 for every $a\in A,$ \cite{sah1}.  Also $A$ is called left
$\phi$-amenable (left $\phi$-contractible) if there exists an element
$m\in A^{**}\,(m\in A)$ such that $am=\phi(a)m$ and
$\tilde{\phi}(m)=1\,(\phi(m)=1)$ for every $a\in A$, respectively.
For more details on left $\phi$-amenability and left
$\phi$-contractibility see \cite{kan} and \cite{nas}, respectively.

Following theorem is given by authors in \cite{sah3}. They
characterized  approximate biprojectivity of some semigroup
algebras. We apply this theorem in order to  characterize  approximate biprojectivity
of algebras related to the locally compact groups.
\begin{Theorem}\cite[Theorem 3.9]{sah3}\label{app give phi}
Let $A$ be an approximately biprojective Banach algebra with a left
approximate identity (right approximate identity) and let
$\phi\in\Delta(A)$. Then $A$ is left $\phi$-contractible(right
$\phi$-contractible), respectively.
\end{Theorem}
\section{Approximate biprojectivity}
In this section we improve \cite[Theorem 3.5]{sam} or \cite[Theorem
5.3]{Choi} and \cite[Theorem 5.13]{hel} concerning  approximate
biprojectivity of some Banach algebras related to the locally
compact groups.

We remind that a Banach algebra $A$ is called   pseudo-contractible
 if there is a not necessarily bounded net $(m_{\alpha})_{\alpha}$ in
$A\otimes_{p}A$ such that $a\cdot m_{\alpha}=m_{\alpha}\cdot a$ and
$\pi_{A}(m_{\alpha})a\rightarrow a$ for each $a\in A$. For the fundamental details of
 the pseudo-contractibility  readers are referred to
\cite{ghah pse} and \cite{Choi}.

Now we consider Segal algebras on   a locally compact group. As we see in \cite{rei} a Segal algebra
$S(G)$  on   a locally compact group $G$ is a dense left ideal of  $L^{1}(G)$ that
satisfies the following conditions:
\begin{enumerate}
\item [(i)]  $S(G)$  is
a Banach space  with respect to a norm $||\cdot||_{S}$ satisfying $|| \cdot||_{L^1}\leq|| \cdot||_{S}$.
\item [(ii)] For $f\in S(G)$ and $y\in G$, $\delta_{y}\ast f\in S(G)$
and the map $y\mapsto \delta_{y}\ast f$ is continuous. Also
$||\delta_{y}\ast f||_{S}=|| f||_{S}$, for $f\in S(G)$ and
$y\in G$.
\end{enumerate}
With the norm $||\cdot||_{S}$ and the convolution product, $S(G)$ is a Banach algebra and we have the following
inequality
$$||f\ast g||_S\leq ||f||_{L^1}||g||_S\qquad f\in L^{1}(G),\,g\in S(G). $$
$S(G)$  on   a locally compact group $G$ always has a left approximate identity and it is never amenable unless it is
 $ L^{1}(G)$ itself and $G$ is amenable.
A Segal algebra is called symmetric if for $f\in S(G)$ and $y\in G$, $f\ast\delta_{y}\in S(G)$ and the map $y\mapsto f\ast \delta_{y}$ is continuous and also $|| f\ast \delta_{y}||_{S}=|| f||_{S}$. By \cite[Proposition 1, page 19]{rei} a symmetric Segal algebra is an ideal in $L^{1}(G)$ which $$||g\ast f||_S\leq ||f||_{L^1}||g||_S\qquad f\in L^{1}(G),\,g\in S(G). $$

Note that  $\Delta(S(G))=\{\phi_{|_{S(G)}}|\phi\in\Delta(L^{1}(G))\}$
and $\phi_{0}$ (the augmentation character on $L^{1}(G)$)
induces a character on $S(G)$ still denoted by $\phi_{0}$
\cite[Lemma 2.2]{alagh}.

 Samei {\it et al.} in
\cite[Theorem 3.5]{sam} and Choi {\it et al.} in \cite[Theorem
5.3]{Choi} showed that $S(G)$ is pseudo-contractible if and only if
$G$ is compact.
As  approximate
biprojectivity is weaker notion  than pseudo-contractibility  in  the following theorem we extend this result.
\begin{Theorem}\label{segal}
Let $G$ be a locally compact group. Then $S(G)$ is approximately
biprojective if and only if $G$ is compact.
\end{Theorem}
\begin{proof}
 Let $S(G)$
be approximately biprojective.  Since  $S(G)$ has a left approximate identity, Theorem \ref{app
give phi} shows that  $S(G)$ is left $\phi_{0}$-contractible, hence by
\cite[Theorem 2.1]{nas} there exists an element $m\in S(G)$ such
that $a\ast m=\phi_{0}(a)m$ and $ \phi_{0}(m)=1$ for every $a\in S(G)$.
Since $S(G)$ is dense in $L^{1}(G)$, it is easy to see that
$a\ast m=\phi_{0}(a)m$ and $ \phi_{0}(m)=1$ for every $a\in L^{1}(G)$.  Now apply
\cite[Theorem 6.1]{nas} to show that $G$ is compact.

Converse is clear by \cite[Theorem 3.5]{sam} or \cite[Theorem
5.3]{Choi}.
\end{proof}

\begin{Theorem}
Let $G$ be an $SIN$ group. If $S(G)\otimes_{p}S(G)$ is approximately
biprojective, then $G$ is compact.
\end{Theorem}
\begin{proof}
The main result of  \cite{kot} asserts that, if $G$ is an $SIN$
group, then $S(G)$ has a central approximate identity, say
$(e_{\alpha})_{\alpha\in I}$.  Since $S(G)\otimes_{p}S(G)$ is
approximately biprojective, there exists a net
$$(\rho_{\beta})_{\beta\in \Theta}:S(G)\otimes_{p}S(G)\rightarrow
(S(G)\otimes_{p}S(G))\otimes_{p}(S(G)\otimes_{p}S(G))$$ of
continuous $S(G)\otimes_{p}S(G)$-bimodule morphism such that
$\pi_{S(G)\otimes_{p}S(G)}\circ\rho_{\beta}(x)\rightarrow x$, for
every $x\in S(G)\otimes_{p}S(G)$.
Consider
$n_{\alpha}=e_{\alpha}\otimes e_{\alpha}$, it is easy to see that  for
every $x\in S(G)\otimes_{p}S(G)$ we have
$xn_{\alpha}=n_{\alpha}x$ and $\phi\otimes
\phi(n_{\alpha})=\phi\otimes\phi(e_{\alpha}\otimes
e_{\alpha})=\phi(e_{\alpha})\phi(e_{\alpha})\rightarrow 1,$
where $\phi\in \Delta(S(G))$.
Define
$m^{\beta}_{\alpha}=\rho_{\beta}(n_{\alpha})$. Then it is easy to
see that $x\cdot m^{\beta}_{\alpha}=m^{\beta}_{\alpha}\cdot x$   for
every $x\in S(G)\otimes_{p}S(G)$. Also
\begin{equation}\label{e2-3}
\begin{split}
\lim_{\alpha}\lim_{\beta}\phi\otimes\phi\circ\pi_{S(G)\otimes_{p}S(G)}(m^{\beta}_{\alpha})-1&=\lim_{\alpha}\lim_{\beta}\phi\otimes \phi\circ \pi_{S(G)\otimes_{p}S(G)}\circ \rho_{\beta}(n_{\alpha})-1\\
&=\lim_{\alpha}\phi\otimes \phi(n_{\alpha})-1\\
&=\lim_{\alpha}\phi(e_{\alpha})^{2}-1=0.
\end{split}
\end{equation}
Set $E=I\times \Theta^{I}$, where $\Theta^{I}$ is the set of
all functions from $I$ into $\Theta$. Consider the product ordering on $E$ as follows  $$(\alpha,\beta)\leq_{E}
(\alpha^{'},\beta^{'})\Leftrightarrow \alpha\leq_{I} \alpha^{'},
\beta\leq_{\Theta^{I}}\beta^{'}\qquad (\alpha,\alpha^{'}\in I,\quad
\beta,\beta^{'}\in \Theta^{I}),$$ here  $
\beta\leq_{\Theta^{I}}\beta^{'}$ means that $\beta(d)\leq_{\Theta}
\beta^{'}(d)$ for each $d\in I$. Suppose that
$\gamma=(\alpha,\beta_{\alpha})\in E$ and
$m_{\gamma}=\rho_{\beta_{\alpha}}(n_{\alpha})\in
(S(G)\otimes_{p}S(G))\otimes_{p}(S(G)\otimes_{p}S(G))$. Now using the
iterated limit theorem \cite[page 69]{kel} in (\ref{e2-3}) we obtain
$$\phi\otimes \phi\circ
\pi_{S(G)\otimes_{p}S(G)}(m_{\gamma})\rightarrow 1$$ and similarly we obtain    $x\cdot m_{\gamma}=m_{\gamma}\cdot x$ for every $x\in
S(G)\otimes_{p}S(G)$ . By using the same argument as in the proof of
\cite[Proposition 2.2]{sah1} one can show that  $S(G)\otimes_{p}S(G)$ is left
$\phi\otimes\phi$-contractible. Hence \cite[Theorem 3.14]{nas} shows that
$S(G)$ is left $\phi$-contractible. So $L^{1}(G)$ is left $\phi$-contractible. Applying \cite[Theorem 6.1]{nas}   $G$  must be  compact.
\end{proof}
Let $G$ be a locally compact group. A real-valued function $w$ on $G$ is said to be a weight function if it has the following properties:
\begin{enumerate}
\item [(i)]  $w(x)\geq 1 \quad (x\in G)$,
\item [(ii)] $w(xy)\geq w(x)w(y)\quad (x,y\in G)$,
\item [(iii)] $w$ is measurable and locally bounded.
\end{enumerate}
 We form the
Banach space
 $$L^{1}(G,w)=\set{f:G\to\mathbb{C}:\, fw\in L^{1}(G)}.$$
    Then $L^{1}(G,w)$, with the convolution product, is a
Banach algebra and is called  Beurling algebra. See \cite{steg} for further information on Beurling algebras.

Helemskii \cite[Theorem 5.13]{hel} showed that the group algebra $L^{1}(G)$ is biprojective
if and only if $G$ is compact. At the
following theorem we extend this result.
\begin{Theorem}
Let $G$ be a locally compact group and let $w$ be a continuous
weight on $G$. Then $L^{1}(G,w)$ is approximately biprojective  if
and only if $G$ is compact.
\end{Theorem}
\begin{proof}
 Let $L^{1}(G,w)$ be
approximately biprojective. Since $L^{1}(G,w)$ has a left approximate
approximate  identity \cite[Propostion 3.7.7]{steg},  by Theorem \ref{app give phi}
$L^{1}(G,w)$ is  left $\phi$-contractible for every
$\phi\in\Delta(L^{1}(G,w))$ even for the augmentation character
$\phi_{0}$ which is specified by
$$\phi_{0}(f)=\int_{G} f(x)dx.$$ By
\cite[Theorem 2.1]{nas} there exists an element $m\in L^{1}(G,w)$
such that $a*m=\phi_{0}(a)m$ and $\phi_{0}(m)=1$ for every $a\in
L^{1}(G,w).$ Pick $f\in L^{1}(G,w)$ such that $\phi_{0}(f)=1.$ We
have
$$\delta_{g}\ast m=\phi_{0}(f)\delta_{g}\ast m=\delta_{g}\ast(f\ast m)=(\delta_{g}\ast f)\ast m=\phi_{0}(\delta_{g}\ast f)m=\phi_{0}(f)m=m,$$
which shows that m is a constant function in $L^{1}(G,w)$, so we can
assume that $1\in L^{1}(G,w)$. Since $w(g)\geq 1$ for every $g\in
G$, we have
$$|G|=\int _{G}1 dg\leq \int_{G}w(g)dg<\infty.$$ Now apply
\cite[Theorem 15.9]{hew} to show that $G$ is compact.

For converse, using the same argument as in \cite[Theorem
5.13]{hel}, it is easy to see that $L^{1}(G,w)$ is biprojective, so
$L^{1}(G,w)$ is approximately biprojective.
\end{proof}
\begin{Proposition}\label{prop1}
Let $G$ be a locally compact group and let $A$ be a unital Banach
algebra with $\Delta(A)\neq \emptyset$. If $A\otimes_{p}L^{1}(G)$ is
approximately biprojective, then $G$ is compact and $A$ is
approximately biprojective. Converse holds if $A$ is biprojective.
\end{Proposition}
\begin{proof}
Suppose that $B=A\otimes_{p}L^{1}(G)$ is  approximate biprojective.
It is easy to see that $(e_{A}\otimes e_{\alpha})$ is an approximate
identity for $B$, where  $e_{A}$ is an identity for $A$ and $(e_{\alpha})$ is
a bounded approximate identity for $L^{1}(G)$. Let
$\psi\in\Delta(A)$ and $\phi\in\Delta(L^{1}(G))$. Then
Theorem \ref{app give phi} implies that  $B$ is left
$\psi\otimes\phi$-contractible. By \cite[Theorem 3.14]{nas}
$L^{1}(G)$ is left $\phi$-contractible which implies that $G$ is
compact, see \cite[Theorem 6.1]{nas}.

Let $\rho:G\rightarrow
\mathbb{C}$ be a group character correspond to $\phi$, see
\cite[Theorem 23.7]{hew}. It is easy to see that $\rho\in
L^{\infty}(G)$. Since $G$ is compact, $L^{\infty}(G)\subseteq
L^{1}(G)$. Then $\rho\in L^{1}(G)$. Also, since
$\rho*f=f*\rho=\phi(f)\rho$ for every $f\in L^{1}(G)$. One can
easily see that $\rho$  is an  idempotent in $L^{1}(G)$. Now
by similar argument as in \cite[Proposition 2.6]{rams}, one can easily see that  $A$ is approximately biprojective.

 Conversely, it is well-known that $L^{1}(G)$ is biprojective if
and only if $G$ is compact. Now apply \cite[Proposition 2.4]{rams},
to complete the proof.
\end{proof}
We recall that a Banach algebra $A$ is left character contractible, if $A$ is left $\phi$-contractible for every $\phi\in\Delta(A)\cup\{0\}$, for more information on this notion, see \cite{nas}.
\begin{Proposition}
Let $G$ be a locally compact group. Then the followings are equivalent:
\begin{enumerate}
\item [(i)] $L^{1}(G)\otimes_{p}M(G)$ is biprojective;
\item [(ii)] $L^{1}(G)\otimes_{p}M(G)$ is approximately biprojective;
\item [(iii)] $G$ is finite.
\end{enumerate}
\end{Proposition}
\begin{proof}
(i)$\Rightarrow$(ii) is clear.

(ii)$\Rightarrow$(iii) Suppose that  $L^{1}(G)\otimes_{p}M(G)$ is
approximately biprojective. Since  $M(G)$ is unital, Proposition
\ref{prop1} shows that $M(G)$ is approximately biprojective. So by Theorem
\ref{app give phi} $M(G)$ is left $\phi$-contractible for
every $\phi\in\Delta(M(G))$. Also by \cite[Proposition 3.4]{nas} $M(G)$ is  left
$0$-contractible. Hence $M(G)$ is left character contractible.
Therefore by \cite[Corollary 6.2]{nas}, $G$ is finite.

(iii)$\Rightarrow$(i) is clear.
\end{proof}
\begin{Proposition}
Let $G$ be an amenable  locally compact group. If  $L^{1}(G)\otimes_{p}\mathcal{A}(G)$ is approximately biprojective, then $G$ is finite.
\end{Proposition}
\begin{proof}
It is well-known that $L^{1}(G)$ has a bounded approximate identity
and by Leptin's theorem amenability of $G$ implies that
$\mathcal{A}(G)$ has a bounded approximate identity, see
\cite[Theorem 7.1.3]{run}. Therefore
$L^{1}(G)\otimes_{p}\mathcal{A}(G)$ has a bounded approximate
identity. Suppose that $L^{1}(G)\otimes_{p}\mathcal{A}(G)$ is
approximately biprojective. Then by  Theorem \ref{app give phi},
$L^{1}(G)\otimes_{p}\mathcal{A}(G)$ is left
$\phi\otimes\psi$-contractible for every $\phi\in\Delta(L^{1}(G))$
and $\psi\in\Delta(\mathcal{A}(G))$. Now by \cite[Theorem
3.14]{nas}, $L^{1}(G)$ is left $\phi$-contractible and
$\mathcal{A}(G)$ is left $\psi$-contractible. By  \cite[Proposition
6.6]{nas}, $G$ is discrete and by \cite[Proposition  6.1]{nas} $G$
is compact, therefore $G$ must be finite.
\end{proof}

\section{$\phi$-biflatness}
In \cite{sah1}, the  authors  studied  $\phi$-biflatness of group algebras.
 In this section we continue the  study of
$\phi$-biflatness of Segal
algebras and the second duals of group algebras.
\begin{Theorem}\label{bif}
Let $A$ be a Banach algebra with a left approximate identity and let $\phi\in\Delta(A)$. If $A$ is $\phi$-biflat, then $A$ is left $\phi$-amenable.
\end{Theorem}
\begin{proof}
Let  $A$ be a $\phi$-biflat Banach algebra. Then
there exists a bounded $A$-bimodule morphism
$\rho:A\rightarrow(A\otimes_{p}A)^{**}$ such that
$\tilde{\phi}\circ\pi^{**}_{A}\circ\rho(a)=\phi(a)$ for every $a\in A.$
 Set
$g=(id_{A}\otimes \overline{\phi})^{**}\circ(id_{A}\otimes
q)^{**}\circ\rho:A\rightarrow (A\otimes_{p}\mathbb{C})^{**}$, where
$L=\ker\phi$, $q:A\rightarrow \frac{A}{L}$ is the quotient map and
$\overline{\phi}:\frac{A}{L}\rightarrow \mathbb{C}$ is a character
defined by $\overline{\phi}(a+L)=\phi(a)$ for every $a\in A.$ We see
that $g$ is a bounded  left $A$-module morphism. We claim that
$g(l)=0$ for every $l\in L.$ Since $A$ has a left approximate
identity, $\overline{AL}=L$. Then for each $l\in L$ there exist
sequences $(a_{n})\subseteq A$ and $(l_{n})\subseteq L$    such that
$a_nl_{n}\rightarrow l$. For $b\in L$, define a map
$R_{b}:A\rightarrow L$ by $R_{b}(a)=ab$ for every $a\in A$. Since $
q \circ R_{l_{n}}=0$, we have
\begin{equation*}
\begin{split}
g(l)=(id_{A}\otimes \overline{\phi})^{**}\circ(id_{A}\otimes q)^{**}(\rho(l)) &=\lim_{n}(id_{A}\otimes
\overline{\phi})^{**}\circ(id_{A}\otimes
q)^{**}(\rho(a_nl_n))\\
&=\lim_{n}(id_{A}\otimes
\overline{\phi})^{**}\circ(id_{A}\otimes q)^{**}(\rho(a_n)\cdot l_n)\\
&=\lim_{n}(id_{A}\otimes
\overline{\phi})^{**}\circ(id_{A}\otimes q)^{**}\circ(id_{A}\otimes R_{l_{n}})^{**}(\rho(a_n))\\
&=\lim_{n}((id_{A}\otimes \overline{\phi})\circ(id_{A}\otimes
q)\circ(id_{A}\otimes
R_{l_{n}}))^{**}(\rho(a_n))\\
&=\lim_{n}((id_{A}\otimes \overline{\phi})\circ( id_{A}\otimes (q \circ
R_{l_{n}}))^{**}(\rho(a_n))=0.
\end{split}
\end{equation*}
 Therefore $g$
induce a map $\overline{g}:\frac{A}{L}\rightarrow (A\otimes_{p}\mathbb{C})^{**}$
which is defined by $\overline{g}(a+L)=g(a)$ for all $a\in A.$ It is easy to  see that  $\overline{g}$
 is a bounded left $A$-module
morphism. Pick $a_{0}$ in $A$ such that $\phi(a_{0})=1$. We denote
$\lambda:A\otimes_{p} \mathbb{C}\rightarrow A$ for a map which is
specified by $\lambda(a\otimes z)=az$ for every $a\in A$ and $z\in
\mathbb{C}$. Set $m=\lambda^{**}\circ \overline{g}(a_{0}+L)\in A^{**}$, we
claim that $am=\phi(a)m$ and $\tilde{\phi}(m)=1$ for every $a\in A.$
Since $\lambda^{**}$  is a left $A$-module morphism and  also
since $aa_{0}+L=\phi(a)a_{0}+L$, we have
\begin{equation}\label{1}
\begin{split}
am=a \lambda^{**}\circ \overline{g}(a_{0}+L)= \lambda^{**}\circ
\overline{g}(aa_{0}+L)=\lambda^{**}\circ
\overline{g}(\phi(a)a_{0}+L)=\phi(a)\lambda^{**}\circ \overline{g}(a_{0}+L)=\phi(a)m
\end{split}
\end{equation}
for every $a\in A.$  Since $\rho(a_{0})\in (A\otimes_{p}A)^{**}$, by Goldestine's theorem there exists a net $(a_{\alpha})$ in
$A\otimes_{p}A$ such that $a_{\alpha}\xrightarrow{w^{*}}\rho(a_{0})$. So
\begin{equation}\label{2}
\begin{split}
\tilde{\phi}(m)=m(\phi) =[ \lambda^{**}\circ \overline{g}(a_{0}+L)](\phi)
&=[\lambda^{**}\circ g(a_{0})](\phi)\\
&=[\lambda^{**}\circ(id_{A}\otimes \overline{\phi})^{**}\circ(id_{A}\otimes
q)^{**}(\rho(a_{0}))](\phi)\\
&=[(\lambda\circ(id_{A}\otimes \overline{\phi})\circ(id_{A}\otimes
q))^{**}(\rho(a_{0}))](\phi)\\
&=[w^{*}-\lim(\lambda\circ(id_{A}\otimes \overline{\phi})\circ(id_{A}\otimes
q))^{**}(a_{\alpha}))](\phi)\\
&=\lim(\lambda\circ(id_{A}\otimes \overline{\phi})\circ(id_{A}\otimes
q))^{**}(a_{\alpha})(\phi)\\
&=\lim(\lambda\circ(id_{A}\otimes \overline{\phi})\circ(id_{A}\otimes
q)(a_{\alpha})(\phi)\\
&=\lim\phi\circ\lambda\circ(id_{A}\otimes \overline{\phi})\circ(id_{A}\otimes
q)(a_{\alpha})\\
&=\lim\phi\circ\pi_{A}(a_{\alpha}).
\end{split}
\end{equation}
On the other hand since $a_{\alpha}\xrightarrow{w^{*}} \rho(a_{0})$,
the $w^{*}$-continuity of $\pi^{**}_{A}$ implies that
$$\pi_{A}(a_{\alpha})=\pi^{**}_{A}(a_{\alpha})\xrightarrow{w^{*}}\pi^{**}_{A}(\rho(a_{0})).$$ Thus
\begin{equation}\label{3}
\begin{split}
\phi(\pi_{A}(a_{\alpha}))=\pi_{A}(a_{\alpha})(\phi)=\pi^{**}_{A}(a_{\alpha})(\phi)\rightarrow\pi^{**}_{A}(\rho(a_{0}))(\phi)=\tilde{\phi}\circ \pi^{**}_{A}(\rho(a_{0}))=1.
\end{split}
\end{equation}
We see that  from (\ref{2}) and (\ref{3}), $\tilde{\phi}(m)=1$.
Combine this result with (\ref{1}), implies that $A$ is left
$\phi$-amenable.
\end{proof}
\begin{cor}\label{bi}
If $S(G)$ is $\phi$-biflat. Then $G$ is amenable
\end{cor}
\begin{proof}
Since every Segal algebra has a left approximate identity, by the
previous Theorem $S(G)$ is left $\phi$-amenable. Then
\cite[Corollary 3.4]{alagh} implies that $G$ is amenable.
\end{proof}
We show that the converse of  Theorem \ref{bif} is also valid  for symmetric Segal algebras.
\begin{Proposition}
Let $G$ be a  locally compact group, and $S(G)$ be a  symmetric Segal algebra on $G$. Then the followings are equivallent
\begin{enumerate}
\item [(i)] $G$ is amenable,
\item [(ii)] $S(G)$ is $\phi$-biflat,
\item [(iii)] $S(G)$ is left $\phi$-amenable.
\end{enumerate}
\end{Proposition}
\begin{proof}
(i)$\Rightarrow$(ii) Let $G$ be an amenable group. Then $L^{1}(G)$ is amenable. So there exists a bounded net $(m_{\alpha})$ in $L^{1}(G)\otimes_{p} L^{1}(G)$ such that $a\cdot m_{\alpha}-m_{\alpha}\cdot a\rightarrow 0$ and $\pi_{L^{1}(G)}(m_{\alpha})a\rightarrow a$ for every $a\in L^{1}(G).$ It is easy to see that   $\phi \circ\pi_{L^{1}(G)}(m_{\alpha})\rightarrow 1$ for every $\phi\in\Delta(L^{1}(G))$. Fix $\phi\in\Delta(L^{1}(G))$. Define a map $R:L^{1}(G)\otimes_{p}L^{1}(G)\rightarrow L^{1}(G)$  by $R(a\otimes b)=\phi(b)a$ and set $L:L^{1}(G)\otimes_{p}L^{1}(G)\rightarrow L^{1}(G)$ for a map which is specified by $L(a\otimes b)=\phi(a)b$ for every $a,b\in L^{1}(G).$ It is easy to see that $L$ and $R$ are bounded linear maps which satisfy
$$L(m\cdot a)=L(m)\ast a,\quad L(a\cdot m)=\phi(a)L(m)\quad(a\in L^{1}(G), m\in L^{1}(G)\otimes_{p}L^{1}(G))$$
and
 $$R(a\cdot m)=a\ast R(m)\quad R(m\cdot a)=\phi(a)R(m)\quad(a\in L^{1}(G), m\in L^{1}(G)\otimes_{p}L^{1}(G)).$$
Thus $$L(m_{\alpha})\ast a-\phi(a)L(m_{\alpha})=L(m_{\alpha}\cdot a-a\cdot m_{\alpha})\rightarrow 0,$$
similarly we have $a\ast R(m_{\alpha})-\phi(a)R(m_{\alpha})\rightarrow 0$ for every $a\in L^{1}(G)$.
Since $$\phi \circ L=\phi\circ R=\phi\circ\pi_{L^{1}(G)},$$ it is easy to see that
$$\phi \circ L(m_{\alpha})=\phi\circ R(m_{\alpha})=\phi\circ\pi_{L^{1}(G)}(m_{\alpha})\rightarrow 1.$$
Pick an element $i_{0}$ in $S(G)$ such that $\phi(i_{0})=1.$ Set $n_{\alpha}=R(m_{\alpha})i_{0}\otimes i_{0}L(m_{\alpha})$ for every $\alpha$. Since
$( L(m_{\alpha}))$ and $(R(m_{\alpha}))$ are bounded nets in $L^{1}(G)$ and since $S(G)$ is an ideal of $L^{1}(G)$, we see that $(n_{\alpha})$ is a bounded net in $S(G)\otimes_{p}S(G)$. Also
\begin{equation}\label{eqqq}
\begin{split}
||a\cdot n_{\alpha}-n_{\alpha}\cdot a||_{S\otimes_{p}S}&=||a\cdot n_{\alpha}-\phi(a)n_{\alpha}+\phi(a)n_{\alpha}-n_{\alpha}\cdot a||_{S\otimes_{p}S}\\
&=||a\cdot n_{\alpha}-\phi(a)n_{\alpha}||_{S\otimes_{p}S}+||\phi(a)n_{\alpha}-n_{\alpha}\cdot a||_{S\otimes_{p}S}\rightarrow 0\quad (a\in S(G))
\end{split}
\end{equation}
and
\begin{equation}\label{eqqq1}
\begin{split}
\phi\circ\pi_{S(G)}(n_{\alpha})=\phi (R(m_{\alpha})\ast i_{0}^{2}\ast L(m_{\alpha}))=\phi(R(m_{\alpha}))\phi(L(m_{\alpha}))\rightarrow 1.
\end{split}
\end{equation}
Let $N$ be a $w^{*}$-cluster point of $(n_{\alpha})$ in $(S(G)\otimes_{p}S(G))^{**}$. Combining  (\ref{eqqq}) and (\ref{eqqq1}) with the  facts $$a\cdot n_{\alpha}\xrightarrow{w^{*}}a\cdot N,\quad n_{\alpha}\cdot a\xrightarrow{w^{*}}N\cdot a,\quad \pi^{**}_{S(G)}(n_{\alpha})\xrightarrow{w^{*}}\pi^{**}_{S(G)}(N)\quad (a\in(S(G))$$ we have
$$a\cdot N=N\cdot a,\quad \tilde{\phi}\circ\pi^{**}_{S(G)}(N)=1\quad (a\in(S(G))).$$
Define a map $\rho:S(G)\rightarrow (S(G)\otimes_{p}S(G))^{**}$ by $\rho(a)=a\cdot N$ for every $a\in S(G).$ It is easy to see that $\rho$ is a bounded $S(G)$-bimodule morphism and $\tilde{\phi}\circ\pi^{**}_{S(G)}\circ\rho(a)=\tilde{\phi}\circ\pi^{**}_{S(G)}(a\cdot N)=\phi(a)$, so $S(G)$ is $\phi$-biflat.\\
(ii)$\Rightarrow$(i)  is clear by Corollary \ref{bi}.\\
(iii)$\Leftrightarrow$(i) is clear by \cite[Corollary 3.4]{alagh}.
\end{proof}
Let $A$ be a Banach algebra and $\phi\in\Delta(A)$.  $A$ is called
$\phi$-inner amenable if there exists an element $m\in A^{**}$ such that $m(f\cdot a)=m(a\cdot f)$ and $\tilde{\phi}(m)=1$ for every $a\in A$ and $f\in A^{*}$, see \cite{jab}. Note that by \cite[Corollay 2.2]{jab} every Banach algebra with a bounded approximate identity is $\phi$-inner amenable.
\begin{Theorem}\label{dual}
Let $A$ be a $\phi$-inner amenable Banach algebra, where
$\phi\in\Delta(A)$. If $A^{**}$ is  $\tilde{\phi}$-biflat, then $A$ is left $\phi$-amenable.
\end{Theorem}
\begin{proof}
Let $A^{**}$ be $\tilde{\phi}$-biflat. Then there
exists a bounded $A^{**}$-bimodule morphism $\rho:A^{**}\rightarrow
(A^{**}\otimes_{p}A^{**})^{**}$  such that for every $a\in A^{**}$
$$\tilde{\tilde{\phi}}\circ\pi^{**}_{A^{**}}\circ\rho(a)=\tilde{\phi}(a),$$
where $\tilde{\tilde{\phi}}$ is an extension of $\tilde{\phi}$ on ${A^{****}}$ as we mentioned in the introduction.
Suppose that  $A$ is  $\phi$-inner amenable. Then  there exists an element $m\in A^{**}$ such that $m(f\cdot a)=m(a\cdot f)$ and $\tilde{\phi}(m)=1$ for every $a\in A$ and $f\in A^{*}$. Set $M=\rho(m)$, since $\rho$ is a bounded $A^{**}$-bimodule morphism,
we have $a\cdot M=M\cdot a$ and  $\tilde{\tilde{\phi}}\circ\pi^{**}_{A^{**}}(M)=\tilde{\phi}(m)=1$ for every $a\in A.$

Now take $\epsilon>0$ and a finite set $F=\{a_{1},...,a_{r}\}\subseteq A$, and set
\begin{equation*}
\begin{split}
V=&\{(a_{1}\cdot n-n\cdot a_{1},..., a_{r}\cdot n-n\cdot a_{r},\tilde{\phi}\circ\pi_{A^{**}}(n)-1)
:n\in A^{**}\otimes_{p}A^{**}, ||n||\leq ||M||\}\\
&\subseteq \prod^{r}_{i=1}(A^{**}\otimes_{p}A^{**})\oplus_{1} \mathbb{C}.
\end{split}
\end{equation*}
 Then $V$ is a convex set and so the weak and the norm closures of $V$ coincide. But by Goldestine's theorem there exists a  net
$(n_{\alpha})\subseteq A^{**}\otimes_{p}A^{**}$ such that $n_{\alpha}\xrightarrow{w^{*}}M$ and $||n_{\alpha}||\leq ||M||$. So for every $a\in F$ we have $a\cdot n_{\alpha}-n_\alpha\cdot a\xrightarrow{w} 0$ and $|\tilde{\phi}\circ\pi_{A^{**}}(n_{\alpha})-1|\rightarrow 0$ which shows that $(0,0,....,0)$ is a $||\cdot||$-cluster point of $V$.
 Thus there
exists an element  $n_{(F,\epsilon)}$ in $A^{**}\otimes_{p}A^{**}$
such that
\begin{equation}\label{e-8}
||a_{i}\cdot n_{(F,\epsilon)}-n_{(F,\epsilon)}\cdot a_{i}||<\epsilon,\quad |\tilde{\phi}\circ\pi_{A^{**}}(n_{(F,\epsilon)})-1|<\epsilon
\end{equation}
for every $i\in\set{1,2,\ldots,r}$.
Now we consider a directed set
$$\Delta=\{(F,\epsilon):F {\hbox{ is a finite subset of }}A, \epsilon>0\},$$ with the following order
$$(F,\epsilon)\leq (F^{\prime},\epsilon^{\prime})\Longrightarrow F\subseteq F^{\prime},\quad \epsilon\geq \epsilon^{\prime}.$$
So the equation  (\ref{e-8}) follows that there exists   a bounded  net  $(n_{(F,\epsilon)})_{(F,\epsilon)\in \Delta}$  in $A^{**}\otimes_{p}A^{**}$ such that
$$a\cdot n_{(F,\epsilon)}-n_{(F,\epsilon)} \cdot a\rightarrow 0,\quad \tilde{\phi}\circ\pi_{A^{**}}(n_{(F,\epsilon)})\rightarrow 1$$
for every $a\in A$. By \cite[Lemma 1.7]{ghah} there exists a bounded
linear map $\psi:A^{**}\otimes_{p} A^{**}\rightarrow (A\otimes_{p}
A)^{**}$ such that for $a,b\in A$ and $m\in A^{**}\otimes_{p}
A^{**}$, the following holds
\begin{enumerate}
\item [(i)] $\psi(a\otimes b)=a\otimes b $,
\item [(ii)] $\psi(m)\cdot a=\psi(m\cdot a)$,\qquad
$a\cdot\psi(m)=\psi(a\cdot m),$
\item [(iii)] $\pi_{A}^{**}(\psi(m))=\pi_{A^{**}}(m).$
\end{enumerate}
 Define $\xi_{(F,\epsilon)}=\psi(n_{(F,\epsilon)})$ which is a net in $(A\otimes_{p}A)^{**}$ and by the previous properties of $\psi$ it satisfies
$$a\cdot \xi_{(F,\epsilon)}-\xi_{(F,\epsilon)} \cdot a\rightarrow 0,\quad \tilde{\phi}\circ\pi^{**}_{A}(\xi_{(F,\epsilon)})\rightarrow 1\quad (a\in A).$$
Now by applying a similar method as we obtained a net from $M$ at
the beginning of the proof, one can obtain a bounded net
$(\gamma_{(F,\epsilon)})_{(F,\epsilon)\in \Delta}$ related to
$\xi_{(F,\epsilon)}$ in $A\otimes_{p}A$ such that
$$a\cdot \gamma_{(F,\epsilon)}-\gamma_{(F,\epsilon)} \cdot a\rightarrow 0,\quad \phi\circ\pi_{A}(\gamma_{(F,\epsilon)})\rightarrow 1\quad (a\in A).$$
Now define
$T:A\otimes_{p}A\rightarrow A$ by $T(a\otimes b)=\phi(b)a$
for every $a$ and $b$ in $A$. It is easy to see that $T$ is a bounded linear map with the following properties $$T(a\cdot
m)=a T(m),\quad T(m\cdot a)=\phi(a)T(m)\qquad (m\in
A\otimes_{p}A,\quad a\in A).$$
Define $\nu_{(F,\epsilon)}=T( \gamma_{(F,\epsilon)})$, it is easy to see that $\nu_{(F,\epsilon)}$ is a bounded net and
$$a\nu_{(F,\epsilon)}-\phi(a)\nu_{(F,\epsilon)}\rightarrow 0,\quad \phi\circ T(\nu_{(F,\epsilon)})=\phi\circ\pi_{A}(\gamma_{(F,\epsilon)})\rightarrow 1\quad (a\in A).$$
Therefore by \cite[Theorem 1.4]{kan} $A$ is left $\phi$-amenable.
\end{proof}
\begin{cor}
Let $G$ be a locally compact group. If $L^{1}(G)^{**}$ is $\tilde{\phi}$-biflat, then $G$ is amenable.
\end{cor}
\begin{proof}
 Since  $L^{1}(G)$ has a bounded approximate identity,  $L^{1}(G)$ is $\phi$-inner amenable.
Thus by Theorem \ref{dual}, $L^{1}(G)$ is left $\phi$-amenable.
Now by  \cite[Corollary 3.4]{alagh} $G$ is amenable.
\end{proof}
\begin{cor}
Let $G$ be a locally compact group and
$\phi,\psi\in\Delta(L^{1}(G))$. If
$(M^{1}(G)\otimes_{p}L^{1}(G))^{**}$ is
$\widetilde{\phi\otimes\psi}$-biflat, then $G$ is amenable.
\end{cor}
\begin{proof}
 We note that  $M(G)\otimes_{p}L^{1}(G)$ has a bounded
 approximate identity and so it is $\phi$-inner amenable. Now by Theorem \ref{dual},
$M(G)\otimes_{p}L^{1}(G)$ is left $\phi\otimes \psi$-amenable, where
$\phi,\psi\in\Delta(L^{1}(G))$. Hence by \cite[Theorem 3.3]{kan},
$L^{1}(G)$ is left $\phi$-amenable, hence  $G$ is amenable.
\end{proof}

\begin{small}

\end{small}

\end{document}